\documentclass[12pt]{article}

\usepackage{amsmath,amssymb,amsbsy,amsfonts,amsthm,latexsym,amsopn,amstext,
                            amsxtra,euscript,amscd}

\numberwithin{equation}{section}
\theoremstyle{plain}
\newtheorem{theorem}{Theorem}[section]

\newtheorem{thm}[theorem]{Theorem}
\newtheorem{lem}[theorem]{Lemma}

\def\({\left(}
\def\){\right)}

\begin{document}

\title{\Large \textbf{A Summation Involving the Divisor and GCD Functions}}
\author{
{Randell  Heyman}  \\
School of Mathematics and Statistics\\ University of New South Wales \\
High Street, Kensington N.S.W. 2052 \\
Australia\\
\text {randell@unsw.edu.au}
}

\date{}
\maketitle

\begin{abstract}
Let $N$ be a positive integer. Dudek asked for an asymptotic formula for the sum of $\tau(\gcd(a,b))$ for all $a$ and $b$ with $ab \le N$. We give an asymptotic result. The approach is partly geometric and differs from the approach used in many recent gcd-sum results.
\end{abstract}

\section{Introduction}
Euclid's lemma states that if $p$ is a prime number and $p$ divides $ab$ then $p$ divides $a$ or $p$ divides $b$ (see, for example, Hardy and Wright \cite[Theorem 3]{Har}). In 2017 Dudek \cite{Dud} quantified the lemma's truth when the prime number requirement is relaxed. He suggested it would be interesting to see an asymptotic formula for
$$S(N):=\sum_{ab \le N}\tau\(\gcd(a,b)\).$$
Let $N$ be a positive integer throughout.
An asymptotic formula does indeed exist as we show in our theorem as follows:
\begin{thm}
\label{thm:main}
Let $N$ be a positive integer. Then
\begin{align*}
S(N)&=\zeta(2)N \log N +((2\gamma-1)\zeta(2)-2\theta)N+O\(\sqrt{N}\),
\end{align*}
where
 $$\theta =\sum_{d < \infty}\frac{\log d}{d^2}.$$
\end{thm}
There is considerable interest in gcd-sum functions (see T\'oth \cite{Tot} and Haukkanen \cite {Hau}  for surveys). We note that results for a related summation,
$$\sum_{a \le N}\(\sum_{b=1}^a\tau(\gcd(a,b))\),$$
can be inferred from the work of Bordell\`es \cite{Bor}. Our theorem is proven differently from the works of Bordell\`es \cite{Bor}, Haukkanen \cite{Hau}, and T\'oth  \cite{Tot} in that we use geometric techniques.

\section{Notation and preparatory lemmas}
We use the cartesian plane with the normal $x$ and $y$ axes. Throughout the term `on and under the curve' will be above but not including the $x$-axis, and to the right but not including the $y$-axis.

For any integer $n\ge 1,$ we let $\tau(n)$ denote the number of divisors of $n$.
As usual the Riemann zeta function is given by
$$\zeta(s)=\sum_{j=1}^\infty \frac{1}{j^s},$$
for all complex numbers $s$ whose real part is greater than 1.
We recall that the notation $f(x) = O(g(x))$   is
equivalent to the assertion that there exists a constant $c>0$ such that $|f(x)|\le c|g(x)|$ for all $x$.
Finally we use $|A|$ to denote the cardinality of a set $A$.

We will require the following lemmas:
\begin{lem}
\label{lem:d divides}
Let $a$ and $b$ be positive integers with $ab \le N$. Then
\begin{align}
\label{eq:d divides both}
\tau(\gcd(a,b))&=\sum_{d \le \sqrt{N}}\sum_{\substack{d|a\\d|b}}1.
\end{align}
\end{lem}
\begin{proof}
Let $\tau(\gcd(a,b))=k$ for some positive integer $k$. So
$$\tau(\gcd(a,b))=|\{d_1,\ldots,d_k : d_i |\gcd(a,b)\}|.$$
If $d_i|\gcd(a,b)$ then, by the properties of the greatest common divisor, $d|a$ and $d|b$.
Therefore $\{d_1,\ldots, d_k : d_i |\gcd(a,b)\} \subseteq \{d : d|a,d|b\}$. Also, if $d|a, d|b$ and $ab \le N$ then $d \le \sqrt{N}$. So
\begin{align}
\label{eq:less than}
\tau(\gcd(a,b))&\le \sum_{d\le \sqrt{N}}\sum_{\substack{d|a\\d|b}}1.
\end{align}
Conversely, suppose that $d_i \in \{d: d|a,d|b\}$. Then $d|a$ and $d|b$. So $d$ divides $\gcd(a,b)$ from the definition of the greatest common divisor and, as before, $d \le \sqrt{N}$. So
$$\{d_1,\ldots d_k : d_i |\gcd(a,b)\} \supseteq\{d: d|a,d|b\}.$$
Therefore
$$\tau(\gcd(a,b))\ge \sum_{d \le \sqrt{N}}\sum_{\substack{d|a\\d|b}}1,$$
which proves the lemma.
\end{proof}
The divisor summatory function have been well studied. For our purposes it will suffice to use the following (see, for example, Hardy and Wright \cite [Notes to Chapter XVIII]{Har}):
\begin{lem}
\label{lem:tau sum}
\begin{align*}
\sum_{x \le N}\tau(x)=N \log N +(2\gamma-1)N +O\(N^\kappa\),
\end{align*}
where $1/4\le\kappa< 1/2$ for sufficiently large $N$.
\end{lem}
The final lemma formalizes the key insight; every point $(x,y)$ on and under the curve $xy \le N/d^2$ contributes exactly 1 to the right hand side of \eqref{eq:integer points}.
\begin{lem}
\label{lem:integer points}
Fix both $N$ and $d \le N$ positive numbers. Then
\begin{equation}
\label{eq:integer points}
\sum_{ab \le N}\sum_{\substack{d|a\\d|b}}1=\sum_{c \le N/d^2}\tau(c).
\end{equation}
\end{lem}
\begin{proof}
Let $$J=\{(a,b):ab \le N , d|a, d|b\}$$ and
$$K=\{r: r|c\text{ for some }c \le N/d^2\}.$$
It will suffice to show that $|J|=|K|$.
Suppose $(a,b) \in J$. So $a=rd$ and $b=sd$ where $r$ and $s$ are both positive integers.
Since $ab \le N$ we have $1 \le  rsd^2 \le N$ and so $1 \le rs \le N/d^2$. So the point $(r,s)$ is an integer point on and under the curve $xy=N/d^2$. Thus $r$ divides some $c$ with $1 \le c \le N/d^2$. So $r \in K$ from which it follows that $|J| \le |K|$. The argument can be reversed. This proves the lemma.
\end{proof}

\section{Proof of Theorem \ref{thm:main}}
Using Lemma \ref{lem:d divides} we have
\begin{align}
\label{eq:d divides a and b}
S(N)&=\sum_{ab \le N} \sum_{d\le \sqrt{N}}\sum_{\substack{d|a\\d|b}}1\notag\\
&= \sum_{d\le \sqrt{N}}\sum_{ab \le N} \sum_{\substack{d|a\\d|b}}1.
\end{align}
From Lemma \ref{lem:integer points} we have, for a fixed $d$, that
$$\sum_{ab \le N}\sum_{\substack{d|a\\d|b}}1=\sum_{c \le N/d^2} \tau(c).$$
Substituting into \eqref{eq:d divides a and b} and then using Lemma \ref{lem:tau sum} we obtain
\begin{align}
\label{eq:main 0}
S(N)&=\sum_{d \le \sqrt{N}}\sum_{c \le N/d^2} \tau(c)\notag\\
&=\sum_{d\le \sqrt{N}}\(\frac{N}{d^2} \log \(\frac{N}{d^2}\) +\frac{(2\gamma-1)N}{d^2} +O\(\(\frac{N}{d^2}\)^{\kappa}\)\)\notag\\
&=\(N\log N+(2\gamma-1)N\)\sum_{d \le \sqrt{N}} \frac{1}{d^2}-2N \sum_{d \le \sqrt{N}} \frac{\log d}{d^2} +\sum_{d \le \sqrt{N}}O\(\(\frac{N}{d^2}\)^{\kappa}\).
\end{align}
We point out that we have introduced some inefficiency here.
We would expect the actual error terms to average out over the summation.
But we have resorted to summing upper bounds.
Next (see, for example, Apostol \cite[Theorem 3.2(b)]{Apo}) we have
$$\sum_{d \le \sqrt{N}}\frac{1}{d^2}=\zeta(2)-\frac{1}{\sqrt{N}}+O\(\frac{1}{N}\).$$
So
\begin{align}
\label{eq:main 1}
\(N\log N+(2\gamma-1)N\)\sum_{d \le N} \frac{1}{d^2}&=\(N \log N+(2\gamma-1)N\)\(\zeta(2)-\frac{1}{\sqrt{N}}+O\(\frac{1}{N}\)\)\notag\\
&=\zeta(2)N \log N-\sqrt{N}\log N+(2\gamma-1)\zeta(2)N+O\(\log N\).
\end{align}
Next, since we have absolute convergence,
\begin{align*}
-2N \sum_{d \le \sqrt{N}} \frac{\log d}{d^2}&=-2N\(\sum_{d < \infty}\frac{\log d}{d^2}-\sum_{d>\sqrt{N}}\frac{\log d}{d^2}\).
\end{align*}
Recall that $$\theta =\sum_{d < \infty}\frac{\log d}{d^2} \text{ throughout.}$$
Then, using Euler's summation formula, we have
\begin{align*}
\sum_{d>\sqrt{N}}\frac{\log d}{d^2}&=\frac{\log N}{2\sqrt{N}}+\frac{1}{\sqrt{N}}+O\(\frac{\log N}{N}\).
\end{align*}
So
\begin{align*}
-2N \sum_{d \le \sqrt{N}} \frac{\log d}{d^2}&=-2N\theta+\sqrt{N}\log N +2\sqrt{N}+O(\log N).
\end{align*}
Finally, using Apostol \cite[Theorem 3.2]{Apo}, we have
\begin{align}
\label{eq:main 2}
\sum_{d \le \sqrt{N}}O\(\(\frac{N}{d^2}\)^{\kappa}\)&=N^{\kappa}O\(\sum_{d \le \sqrt{N}}\frac{1}{d^{2\kappa}}\)\notag\\
&=N^{\kappa}O\(N^{1/2-\kappa}\)\notag\\
&=O\(\sqrt{N}\).
\end{align}
Substituting \eqref{eq:main 1} and \eqref{eq:main 2} into \eqref{eq:main 0} completes the proof.

\section{Acknowledgment}
The author thanks the editor and referee for their suggestions which, amongst other things, led to an improvement in the error term.

\hrule
\bigskip
Keywords: arithmetic function, divisor function, greatest common divisor.

Mathematical Subject Classification 2010: 11N56
\bigskip
\hrule
\bigskip
Concerned with A2268732.
\bigskip
\hrule

\end{document}